\documentclass[12pt]{amsart}%
\usepackage{amscd,amsmath,latexsym,amsthm,amsfonts,amssymb,graphicx,color,geometry}
\usepackage{enumerate}
\usepackage[pdftex,plainpages=false,pdfpagelabels,
colorlinks=true,citecolor=blue,linkcolor=blue,urlcolor=black,
filecolor=black,bookmarksopen=true,unicode=true]{hyperref}
\usepackage[norefs,nocites]{refcheck}
\usepackage{amsmath}
\usepackage{amsfonts}
\usepackage{amssymb}
\usepackage[T1]{fontenc}
\usepackage{graphicx}
\usepackage{xcolor}
\usepackage{colortbl}%
\providecommand{\U}[1]{\protect\rule{.1in}{.1in}}
\newtheorem{theorem}{Theorem}[section]

\newtheorem{example}[theorem]{Example}

\newtheorem{lemma}[theorem]{Lemma}

\newtheorem{definition}[theorem]{Definition}
\numberwithin{equation}{section}
\begin{document}
\title{Pointwise lineability in sequence spaces}

\begin{abstract}
We prove that several results of lineability/spaceability in the framework of
sequence spaces are valid in a stricter sense.

\end{abstract}
\author[D. Pellegrino]{Daniel Pellegrino}
\address{Departamento de Matem\'{a}tica \\
Universidade Federal da Para\'{\i}ba \\
58.051-900 - Jo\~{a}o Pessoa, Brazil.}
\email{daniel.pellegrino@academico.ufpb.br and dmpellegrino@gmail.com}
\author[A. Raposo Jr.]{Anselmo Raposo Jr.}
\address{Departamento de Matem\'{a}tica \\
Universidade Federal do Maranh\~{a}o \\
65085-580 - S\~{a}o Lu\'{\i}s, Brazil.}
\email{anselmo.junior@ufma.br}
\thanks{D. Pellegrino was partially supported by CNPq 307327/2017-5 and Grant
2019/0014 Paraiba State Research Foundation (FAPESQ)}
\keywords{Spaceability; lineability; sequence spaces}
\subjclass[2020]{46B87, 15A03, 47B37, 47L05}
\maketitle

\section{Introduction}

The notions of lineability and spaceability were introduced in the seminal
paper \cite{AGSS} by Aron, Gurariy, and Seoane-Sep\'{u}lveda. If $V$ is a
vector space and $\alpha$ is a cardinal number, a subset $A$ of $V$ is called
$\alpha\text{-lineable}$ in $V$ if $A\cup\left\{  0\right\}  $ contains an
$\alpha$-dimensional linear subspace $W$ of $V$. When $V$ has a topology and
the subspace $W$ can be chosen to be closed, we say that $A$ is $\alpha
$-spaceable. For details on the theory of lineability/spaceability we refer to
\cite{AronBGPSS, BGPSS} and for recent results we refer to \cite{a0,a1,a2,
falco} and the references therein.

In this paper we shall be interested in the following stricter variant of
lineability/spaceability, introduced in \cite{FPT}. Let $\alpha$, $\beta$ and
$\lambda$ be cardinal numbers and $V$ be a vector space, with $\dim V=\lambda$
and $\alpha<\beta\leq\lambda$. A set $A\subset V$ is $\left(  \alpha
,\beta\right)  $-lineable if $A$ is $\alpha$-lineable and for every subspace
$W_{\alpha}\subset V$ with $W_{\alpha}\subset A\cup\left\{  0\right\}  $ and
$\dim W_{\alpha}=\alpha$, there is a subspace $W_{\beta}\subset V$ with $\dim
W_{\beta}=\beta$ and $W_{\alpha}\subset W_{\beta}\subset A\cup\left\{
0\right\}  $. Furthermore, if $W_{\beta}$ can always be chosen as a closed
subspace, we say that $A$ is $\left(  \alpha,\beta\right)  $-spaceable. Notice
that the ordinary notions of lineability and spaceability are recovered when
$\alpha=0$.

A challenging fashion of the investigation of lineability/spaceability is that
there are only few general techniques (see \cite{BGC, KT}) and each problem
seems to need \textit{ad hoc} arguments to be solved. From now on all vector
spaces are considered over a fixed scalar field $\mathbb{K}$ which can be
either $\mathbb{R}$ or $\mathbb{C}$. For any set $A$ we shall denote by
$\operatorname*{card}\left(  A\right)  $ the cardinality of $A$; in
particular, we denote $\mathfrak{c}=\operatorname*{card}\left(  \mathbb{R}%
\right)  $ and $\aleph_{0}=\operatorname*{card}\left(  \mathbb{N}\right)  $.

In some sense, the problems investigated in the framework of $\left(
\alpha,\beta\right)  $-lineability/space\-ability can be divided into three categories:

\begin{itemize}
\item $\left(  1,\beta\right)  $-lineability/spaceability is obtained but the
techniques seem to be not adapted to $\left(  \alpha,\beta\right)
$-lineability/spaceability for $\alpha>1$;

\item $\left(  \alpha,\beta\right)  $-lineability/spaceability is
characterized for all $\alpha<\beta$;

\item No technique is known to obtain anything else than lineability/spaceability.
\end{itemize}

Let us illustrate examples of all three situations:

In \cite[Theorem 3.1]{FPT} it is shown that if $p,q\geq1$, the set
\[
A:=\left\{  u:\ell_{p}\rightarrow\ell_{q}:u\text{ is continuous and not
injective}\right\}
\]
is $\left(  1,\mathfrak{c}\right)  $-lineable in the set of continuous linear
operators from $\ell_{p}$ to $\ell_{q}$. Still in \cite[Theorem 3.2]{FPT}, it
was proved that%
\[
L_{p}\left[  0,1\right]  \backslash%
{\displaystyle\bigcup\limits_{q>p}}
L_{q}\left[  0,1\right]
\]
is $\left(  1,\mathfrak{c}\right)  $-spaceable in $L_{p}\left[  0,1\right]  $
for every $p>0$. Both techniques seem to be not immediately adapted to
$\left(  n,\mathfrak{c}\right)  $-lineability/spaceability for $n>1$. The
second case is illustrated by the following result of \cite[Theorem 3.2]{FPT}:
the set
\[
\ell_{p}\backslash\bigcup_{0<q<p}\ell_{q}%
\]
is $\left(  \alpha,\mathfrak{c}\right)  $-spaceable in $\ell_{p}$ if, and only
if, $\alpha<\aleph_{0}$. Finally the third case can be illustrated by results
related to norm attaining operators (see, for instance, \cite[Proposition
6]{PT}).

In this paper we introduce the notion of pointwise lineability/spaceability.
This new concept is closely linked to the notion of $\left(  1,\beta\right)
$-lineability/spaceability but being of a stricter nature. Our main goal is to
develop general techniques in the framework of pointwise
lineability/spaceability in sequence spaces. Lineability and spaceability were
exhaustively investigated in this framework (see \cite{BCFP, BDFP, BF, CSS,
FPT} and the references therein) and in the present paper we shall show
several of these results hold for the more involved notion of pointwise lineability/spaceability.

This paper is organized as follows. In Section $2$ we introduce the notion of
pointwise lineability and preliminary terminology that shall be used
throughout the paper. In Section $3$ we prove our main results, which show
that the main results of \cite{BCFP, BDFP, BF,NP} are valid in the context of
pointwise spaceability.

\section{Pointwise spaceability in sequence spaces}

We start this section establishing the concept of pointwise lineability/spaceability.

\begin{definition}
Let $V$ be an infinite-dimensional vector space over $\mathbb{K}$ and let be
$A$ a non-empty subset of $V$. We say that $A$ is pointwise lineable if for
each $x\in A$ there is an infinite-dimensional subspace $W_{x}$ such that%
\[
x\in W_{x}\subset A\cup\left\{  0\right\}  \text{\emph{.}}%
\]
When $\alpha\leq\dim\left(  V\right)  $ is a cardinal number, we will say that
$A$ is pointwise $\alpha$-lineable if for each $x\in A$ there is a subspace
$W_{x,\alpha}$ such that%
\[
\dim\left(  W_{x,\alpha}\right)  =\alpha\text{ \ \ \ \ and \ \ \ \ }x\in
W_{x,\alpha}\subset A\cup\left\{  0\right\}  \text{.}%
\]
Similarly, we define pointwise spaceability.
\end{definition}

It is obvious that pointwise lineability/spaceability implies $\left(
1,\beta\right)  $-lineability/spaceability. It is also simple to show that the
converse is not true. In fact, $\mathbb{R}^{2}\cup\left\{  (1,1,1)\right\}  $
is not pointwise lineable in $\mathbb{R}^{3}$ but it is $\left(  1,2\right)
$-lineable in $\mathbb{R}^{3}$. The following example is perhaps more
interesting and less artificial.

\begin{example}
Let $X$ be an infinite-dimensional Banach space endowed with weak-topology and
let $A\neq X$ be a basic neighborhood of $0$. Thus, there exist $\varepsilon
>0$, $n\in\mathbb{N}$ and $\varphi_{1},\ldots,\varphi_{n}\in X^{\ast
}\backslash\left\{  0\right\}  $, where $X^{\ast}$ is the topological dual of
$X$, such that%
\[
A=\left\{  x\in X:\left\vert \varphi_{i}\left(  x\right)  \right\vert
<\varepsilon\text{ for all }i=1,\ldots,n\right\}
\]
Note that $\left(
{\textstyle\bigcap\limits_{i=1}^{n}}
\ker\left(  \varphi_{i}\right)  \right)  \subset A$ and
\[
\dim\left(
{\textstyle\bigcap\limits_{i=1}^{n}}
\ker\left(  \varphi_{i}\right)  \right)  =\dim\left(  X\right)  \text{.}%
\]
Let $W$ be a subspace of $X$ such that $W\subset A$ and $\dim\left(  W\right)
=\alpha<\dim\left(  X\right)  $ and let us see that $W\subset\left(
{\textstyle\bigcap\limits_{i=1}^{n}}
\ker\left(  \varphi_{i}\right)  \right)  $. In fact, if there exists $x_{0}\in
W-\left(
{\textstyle\bigcap\limits_{i=1}^{n}}
\ker\left(  \varphi_{i}\right)  \right)  $, then $\varphi_{i}\left(
x_{0}\right)  \neq0$ for a certain $i$. Considering $\lambda:=\dfrac
{\varepsilon+1}{\left\vert \varphi_{i}\left(  x_{0}\right)  \right\vert }$,
since $\lambda x_{0}\in W$ and
\[
\left\vert \varphi_{i}\left(  \lambda x_{0}\right)  \right\vert =\varepsilon
+1>\varepsilon\text{,}%
\]
we have $\lambda x_{0}\notin A$, a contradiction. Therefore, $W\subset\left(
{\textstyle\bigcap\limits_{i=1}^{n}}
\ker\left(  \varphi_{i}\right)  \right)  $ and $A$ is $\left(  \alpha
,\dim\left(  X\right)  \right)  $-spaceable. In particular, $A$ is $\left(
1,\dim\left(  X\right)  \right)  $-spaceable. Since $%
{\textstyle\bigcap\limits_{i=1}^{n}}
\ker\left(  \varphi_{i}\right)  \subsetneq X$, let $y_{0}\in X\backslash
\left(
{\textstyle\bigcap\limits_{i=1}^{n}}
\ker\left(  \varphi_{i}\right)  \right)  $ and
\[
\delta:=\max_{1\leq i\leq n}\left\vert \varphi_{i}\left(  y_{0}\right)
\right\vert >0\text{.}%
\]
For $\alpha:=\left(  2\delta\right)  ^{-1}\varepsilon$, we have $\alpha
y_{0}\notin%
{\textstyle\bigcap\limits_{i=1}^{n}}
\ker\left(  \varphi_{i}\right)  $ and
\[
\left\vert \varphi_{i}\left(  \alpha y_{0}\right)  \right\vert =\alpha
\left\vert \varphi_{i}\left(  y_{0}\right)  \right\vert <\varepsilon\text{,}%
\]
and this shows that $\alpha y_{0}\in A$. But it is obvious that there is no
subspace of $X$ containing $\alpha y_{0}$ and contained in $A$. Therefore, $A$
is not pointwise spaceable.
\end{example}

Many results on lineability in sequence spaces invoke the \textquotedblleft
mother vector\textquotedblright\ technique, where a sequence $\left(
x_{k}\right)  _{k=1}^{\infty}$ (called mother vector) in a sequence space $E$
is used to generate an infinite-dimensional subspace of $E$ contained in a
certain set of sequences. However, in general, the generated subspace does not
contain the \textquotedblleft mother vector\textquotedblright\ and it is by no
means simple to construct the subspace containing the \textquotedblleft mother
vector\textquotedblright. Examples can be found in the papers \cite{FPT, PT}.
The concept of pointwise lineability/spaceability faces this technicality: to
generate a subspace containing the mother vector. Thus, the notion of
pointwise lineability/spaceability shall not be confused with the
\textquotedblleft mother vector\textquotedblright\ technique.

Recently, some results in this direction have been obtained covered under the
veil of $\left(  1,\beta\right)  $-lineability/spaceability. Analyzing the
proof of \cite[Theorem 1.3]{pre} we conclude that, in fact, we have pointwise
lineability/spaceability result.

In order to deal with a wide range of sequence spaces we shall use the notion
of invariant sequence spaces (see Definition \ref{Def3.1.1}) that was
introduced in \cite{BDFP}. For the sake of illustration, in \cite[Theorem
2.5]{BF} it is proved that if $E$ is an invariant sequence space over a Banach
space $X$, then

\begin{enumerate}
\item For every $\Gamma\subset\left(  0,\infty\right]  $, the set
\begin{equation}
E\backslash%
{\displaystyle\bigcup\limits_{q\in\Gamma}}
\ell_{q}\left(  X\right)  \label{r44}%
\end{equation}
is either empty or spaceable;

\item For every $\Gamma\subset\left(  0,\infty\right]  $, the set
\begin{equation}
E\backslash%
{\displaystyle\bigcup\limits_{q\in\Gamma}}
\ell_{q}^{w}\left(  X\right)  \label{s44}%
\end{equation}
is either empty or spaceable, where $\ell_{p}^{w}\left(  X\right)  $ is the
space of weakly $p$-summable sequences in $X$ (the exact definitions shall be
presented throughout the paper).
\end{enumerate}

Let us recall the notion of invariant sequence spaces.

\begin{definition}
\label{Def3.1.1}\emph{(See} \cite[Definition 2.1]{BDFP}\emph{)} Let
$X\neq\left\{  0\right\}  $ be a Banach space over $\mathbb{K}$.

\begin{enumerate}
\item[\emph{(a)}] Given $x=\left(  x_{j}\right)  _{j=1}^{\infty}\in
X^{\mathbb{N}}$, we define the zero-free version of $x$, that we denote by
$x^{0}$, as follows: if $x$ has only finitely many nonzero coordinates, then
$x^{0}=0$, otherwise, $x^{0}=\left(  x_{j_{k}}\right)  _{k=1}^{\infty}$ where
$x_{j_{k}}$ is the $k$-th nonzero coordinate of $x$.

\item[\emph{(b)}] An invariant sequence space over $X$ is an
infinite-dimensional Banach or quasi-Banach space $E$ whose elements are
$X$-valued sequences satisfying the following conditions:

\begin{enumerate}
\item[\emph{(b1)}] For all $x\in X^{\mathbb{N}}$ such that $x^{0}\neq0$,
\[
x\in E\Leftrightarrow x^{0}\in E
\]
and
\[
\left\Vert x\right\Vert \leq K\left\Vert x^{0}\right\Vert
\]
for some constant $K$ which depends only on $E$.

\item[\emph{(b2)}] $\left\Vert x_{j}\right\Vert _{X}\leq\left\Vert
x\right\Vert _{E}$ for all $x=\left(  x_{j}\right)  _{j=1}^{\infty}\in E$ and
all $j\in\mathbb{N}$.
\end{enumerate}

\noindent An invariant sequence space is an invariant sequence space over some
Banach space $X$.
\end{enumerate}
\end{definition}

Usual sequence spaces are invariant sequence spaces. For instance, if $X$ is a
Banach space, the classical sequence spaces $\ell_{p}\left(  X\right)  $,
$\ell_{p}^{w}\left(  X\right)  $, $\ell_{p}^{u}\left(  X\right)  $,
$c_{0}\left(  X\right)  $, $c\left(  X\right)  $, $p\in\left(  0,\infty
\right)  $ and the Lorentz space $\ell_{p,q}$ are invariant sequence spaces.
For more details and examples, we refer to \cite{BDFP}.

Let $X$ and $Y$ be Banach spaces, $\Gamma$ be an arbitrary set and $E$ be an
invariant sequence space over $X$. If $E_{\lambda}$ is an invariant sequence
space over $Y$ for each $\lambda\in\Gamma$, and $f\colon X\rightarrow Y$ is
any function, following \cite[Definition 1.3]{NP}, we define
\[
G\left(  E,f,\left(  E_{\lambda}\right)  _{\lambda\in\Gamma}\right)  =\left\{
\left(  x_{j}\right)  _{j=1}^{\infty}\in E:\left(  f\left(  x_{j}\right)
\right)  _{j=1}^{\infty}\notin%
{\displaystyle\bigcup\limits_{\lambda\in\Gamma}}
E_{\lambda}\right\}  \text{.}%
\]
According to \cite[Definition 2.3]{BF}, a function $f\colon X\rightarrow Y$
between normed spaces is:

\begin{enumerate}
\item[(a)] Non-contractive if $f\left(  0\right)  =0$ and for all scalars
$\alpha$ there is some constant $K\left(  \alpha\right)  >0$ such that%
\[
\left\Vert f\left(  \alpha x\right)  \right\Vert _{Y}\geq K\left(
\alpha\right)  \left\Vert f\left(  x\right)  \right\Vert _{Y}\text{,}%
\]
for all $x\in X$.

\item[(b)] Strongly non-contractive if $f\left(  0\right)  =0$ and, for all
scalars $\alpha$ there exists some constant $K\left(  \alpha\right)  >0$ such
that%
\[
\left\vert \varphi\left(  f\left(  \alpha x\right)  \right)  \right\vert \geq
K\left(  \alpha\right)  \left\vert \varphi\left(  f\left(  x\right)  \right)
\right\vert \text{,}%
\]
for all $x\in X$ and all continuous linear functionals $\varphi\in Y^{\prime}%
$. (From now on, $Y^{\prime}$ denotes the topological dual of $Y$.)
\end{enumerate}

An invariant sequence space $E$ over a Banach space $X$ is called strongly
invariant sequence space when%
\[
c_{00}\left(  X\right)  :=\left\{  \left(  x_{j}\right)  _{j=1}^{\infty}\in
c_{0}\left(  X\right)  :\left(  x_{j}\right)  _{j=1}^{\infty}\text{ is
eventually null}\right\}  \subset E
\]
and $\left(  x_{j}\right)  _{j=1}^{\infty}\in E$ if, and only if, all the
subsequences of $\left(  x_{j}\right)  _{j=1}^{\infty}$ also belong to $E$.

All the aforementioned examples of invariant sequence spaces are also strongly
invariant sequence spaces. However, the notions are not exactly the same. The
set%
\[
E=\left\{  \left(  x_{j}\right)  _{j=1}^{\infty}\in\ell_{1}:%
{\textstyle\sum\limits_{j=1}^{\infty}}
x_{j}=0\right\}  \text{,}%
\]
is the kernel of the continuous linear functional%
\begin{align*}
\varphi\colon\ell_{1}  &  \rightarrow\mathbb{K}\\
x=\left(  x_{j}\right)  _{j=1}^{\infty}  &  \mapsto\varphi\left(  x\right)  =%
{\textstyle\sum\limits_{j=1}^{\infty}}
x_{j}%
\end{align*}
and thus $E$ is a Banach space. It is plain that $\left\vert x_{j}\right\vert
\leq\left\Vert x\right\Vert $ for every $j$ and, if $x^{0}\neq0$ then $x\in
E\Leftrightarrow$ $x^{0}\in E$; it is also plain that $\left\Vert
x^{0}\right\Vert =\left\Vert x\right\Vert $. Hence $E$ is a invariant sequence
space. However, $E$ fails to be a strongly invariant sequence space. In fact,
note that%
\[
\left(  1,-1,2^{-1},-2^{-1},2^{-2},-2^{-2},\ldots,2^{-n},-2^{-n}%
,\ldots\right)
\]
lies in $E$ and, on the other hand, the subsequence
\[
\left(  1,2^{-1},2^{-2},\ldots,2^{-n},\ldots\right)
\]
does not lie in $E$ (it is also simple to check that $c_{00}$ is not contained
in $E$).

According to \cite[Definition 2.3]{NP}, if $X$ and $Y$ are Banach spaces and
$E$ is an invariant sequence space over $Y$, a function $f\colon X\rightarrow
Y$ such that $f\left(  0\right)  =0$ is called compatible with $E$ if for any
$X$-valued sequence $\left(  x_{j}\right)  _{j=1}^{\infty}$ and all scalars
$a\neq0$, we have%
\[
\left(  f\left(  x_{j}\right)  \right)  _{j=1}^{\infty}\notin E\Rightarrow
\left(  f\left(  ax_{j}\right)  \right)  _{j=1}^{\infty}\notin E\text{.}%
\]
For instance, it is easy to see that any non-contractive map $f\colon
X\rightarrow Y$ is compatible with $\ell_{q}\left(  Y\right)  $ and
$c_{0}\left(  Y\right)  $.

\section{Main results}

In this section we shall show that, in general, the results from \cite{BDFP,
BF, NP} also hold in the context of pointwise lineability/spaceability.

The following theorem can be found in \cite[Theorem 2.5]{NP} and it
generalizes \cite[Theorem 2.5(a)]{BF}:

\begin{theorem}
Let $X$ and $Y$ be Banach spaces, $\Gamma$ be an arbitrary set, $E$ be an
invariant sequence space over $X$ and $E_{\lambda}$ be a strongly invariant
sequence space over $Y$ for all $\lambda\in\Gamma$. If $f\colon X\rightarrow
Y$ is compatible with $E_{\lambda}$ for each $\lambda\in\Gamma$, then
$G\left(  E,f,\left(  E_{\lambda}\right)  _{\lambda\in\Gamma}\right)  $ is
either empty or $\mathfrak{c}$-spaceable.
\end{theorem}

The main result of this subsection shows that if $\left(  E_{\lambda}\right)
_{\lambda\in\Gamma}$ is \textquotedblleft nested\textquotedblright, that is,
given $\lambda_{1},\lambda_{2}\in\Gamma$, either $E_{\lambda_{1}}\subset
E_{\lambda_{2}}$ or $E_{\lambda_{1}}\subset E_{\lambda_{2}}$, then we can
assure pointwise spaceability in the above result. We need the following
simple lemma for further reference.

\begin{lemma}
\label{Lema2.2}Let $\left(  E_{\lambda}\right)  _{\lambda\in\Gamma}$ be a
family of nested strongly invariant sequence spaces and let $\mathbb{N}%
^{\prime}\subset\mathbb{N}$ be such that
\[
\operatorname*{card}\left(  \mathbb{N}^{\prime}\right)  =\operatorname*{card}%
\left(  \mathbb{N}\backslash\mathbb{N}^{\prime}\right)  =\aleph_{0}\text{.}%
\]
If $x=\left(  x_{j}\right)  _{j=1}^{\infty}\in X^{\mathbb{N}}$ and $x\notin%
{\displaystyle\bigcup\limits_{\lambda\in\Gamma}}
E_{\lambda}$, then either $\left(  x_{j}\right)  _{j\in\mathbb{N}^{\prime}%
}\notin%
{\displaystyle\bigcup\limits_{\lambda\in\Gamma}}
E_{\lambda}$ or $\left(  x_{j}\right)  _{j\in\mathbb{N}\backslash
\mathbb{N}^{\prime}}\notin%
{\displaystyle\bigcup\limits_{\lambda\in\Gamma}}
E_{\lambda}$.
\end{lemma}

\begin{proof}
Considering the subsequences $\left(  x_{j}\right)  _{j\in\mathbb{N}^{\prime}%
}$, $\left(  x_{j}\right)  _{j\in\mathbb{N}\backslash\mathbb{N}^{\prime}}$,
let us assume that
\[
\left(  x_{j}\right)  _{j\in\mathbb{N}^{\prime}}\in%
{\displaystyle\bigcup\limits_{\lambda\in\Gamma}}
E_{\lambda}\text{ and }\left(  x_{j}\right)  _{j\in\mathbb{N}\backslash
\mathbb{N}^{\prime}}\in%
{\displaystyle\bigcup\limits_{\lambda\in\Gamma}}
E_{\lambda}\text{.}%
\]
Thus, there are $\lambda_{1},\lambda_{2}\in\Gamma$ such that $\left(
x_{j}\right)  _{j\in\mathbb{N}^{\prime}}\in E_{\lambda_{1}}$ and $\left(
x_{j}\right)  _{j\in\mathbb{N}\backslash\mathbb{N}^{\prime}}\in E_{\lambda
_{2}}$ and, without loss of generality, assuming $E_{\lambda_{1}}\subset
E_{\lambda_{2}}$, we have that both $\left(  x_{j}\right)  _{j\in
\mathbb{N}^{\prime}}$ and $\left(  x_{j}\right)  _{j\in\mathbb{N}%
\backslash\mathbb{N}^{\prime}}$ belong to $E_{\lambda_{2}}$. Let us take the
sequences $u=\left(  u_{j}\right)  _{j=1}^{\infty}$ and $v=\left(
v_{j}\right)  _{j=1}^{\infty}$ given by
\[
u_{j}=\left\{
\begin{array}
[c]{ll}%
0\text{,} & \text{\hspace{-0.25cm}if }j\in\mathbb{N}\backslash\mathbb{N}%
^{\prime}\text{,}\\
x_{j}\text{,} & \text{\hspace{-0.25cm}if }j\in\mathbb{N}^{\prime}\text{,}%
\end{array}
\right.  \text{ \ \ \ \ and \ \ \ \ }v_{j}=\left\{
\begin{array}
[c]{ll}%
x_{j}\text{,} & \text{\hspace{-0.25cm}if }j\in\mathbb{N}\backslash
\mathbb{N}^{\prime}\text{,}\\
0\text{,} & \text{\hspace{-0.25cm}if }j\in\mathbb{N}^{\prime}\text{.}%
\end{array}
\right.
\]
We clearly see that $u^{0}=\left[  \left(  x_{j}\right)  _{j\in\mathbb{N}%
^{\prime}}\right]  ^{0}$ and $v^{0}=\left[  \left(  x_{j}\right)
_{j\in\mathbb{N}\backslash\mathbb{N}^{\prime}}\right]  ^{0}$. If $u^{0}=0$,
then $u\in c_{00}\left(  X\right)  \subset E_{\lambda_{2}}$ because
$E_{\lambda_{2}}$ is a strongly invariant sequence space. If $u^{0}\neq0$,
then $u\in E_{\lambda_{2}}$ because, in particular, $E_{\lambda_{2}}$ is an
invariant sequence space. By the same reason we have that $v\in E_{\lambda
_{2}}$. Now, since
\[
x=u+v\text{,}%
\]
and $E_{\lambda_{2}}$ is a vector space, we conclude that%
\[
x\in E_{\lambda_{2}}\subset%
{\displaystyle\bigcup\limits_{\lambda\in\Gamma}}
E_{\lambda}\text{,}%
\]
contradicting the fact that $x\notin%
{\displaystyle\bigcup\limits_{\lambda\in\Gamma}}
E_{\lambda}$.
\end{proof}

In order to fix some notation, if $X$ is a linear space, $x_{k}\in X$ for all
$k\in\mathbb{N}$ and $\left\{  i_{1}<i_{2}<\cdots\right\}  $ is a subset of
$\mathbb{N}$, we denote by
\[%
{\textstyle\sum\limits_{k=1}^{\infty}}
x_{k}\otimes e_{i_{k}}%
\]
the $X$-valued sequence having the $i_{k}$-th coordinate equal to $x_{k}$ and
all the other coordinates are zero.

\begin{theorem}
\label{Teo1}Let $X$ and $Y$ be Banach spaces, $\Gamma$ be an arbitrary set,
$E$ be an invariant sequence space over $X$ and $E_{\lambda}$ be a strongly
invariant sequence space over $Y$ for each $\lambda\in\Gamma$. If $f\colon
X\rightarrow Y$ is compatible with $E_{\lambda}$ for all $\lambda\in\Gamma$
and $\left(  E_{\lambda}\right)  _{\lambda\in\Gamma}$ is nested, then
$G\left(  E,f,\left(  E_{\lambda}\right)  _{\lambda\in\Gamma}\right)  $ is
either empty or pointwise $\mathfrak{c}$-spaceable.
\end{theorem}

\begin{proof}
Let us assume that $G\left(  E,f,\left(  E_{\lambda}\right)  _{\lambda
\in\Gamma}\right)  $ is non-empty and consider
\[
x=\left(  x_{j}\right)  _{j=1}^{\infty}\in G\left(  E,f,\left(  E_{\lambda
}\right)  _{\lambda\in\Gamma}\right)  \text{.}%
\]
Notice that if $\mathbb{N}_{x}=\left\{  j\in\mathbb{N}:x_{j}\neq0\right\}  $,
then $\operatorname*{card}\left(  \mathbb{N}_{x}\right)  =\aleph_{0}$ because
$f\left(  0\right)  =0$ and $c_{00}\left(  Y\right)  \subset E_{\lambda}$ for
each $\lambda\in\Gamma$. So, we can infer that $x^{0}\neq0$. First, we have to
show that%
\[
x^{0}\in G\left(  E,f,\left(  E_{\lambda}\right)  _{\lambda\in\Gamma}\right)
\text{.}%
\]
We know that $\left(  f\left(  x_{j}\right)  \right)  _{j=1}^{\infty}\notin%
{\displaystyle\bigcup\limits_{\lambda\in\Gamma}}
E_{\lambda}$ and, thus, $\left[  \left(  f\left(  x_{j}\right)  \right)
_{j=1}^{\infty}\right]  ^{0}\notin%
{\displaystyle\bigcup\limits_{\lambda\in\Gamma}}
E_{\lambda}$ since, for each $\lambda\in\Gamma$, $E_{\lambda}$ is an strongly
invariant sequence space. Let us denote $x^{0}=\left(  x_{j_{k}}\right)
_{k=1}^{\infty}$, where $x_{j_{k}}$ is the $k$-th non-null coordinate of $x$.
Hence, we have to show that $\left(  f\left(  x_{j_{k}}\right)  \right)
_{k=1}^{\infty}\notin%
{\displaystyle\bigcup\limits_{\lambda\in\Gamma}}
E_{\lambda}$. Since $f\left(  0\right)  =0$, it follows that
\[
\left[  \left(  f\left(  x_{j_{k}}\right)  \right)  _{k=1}^{\infty}\right]
^{0}=\left[  \left(  f\left(  x_{j}\right)  \right)  _{j=1}^{\infty}\right]
^{0}\notin%
{\displaystyle\bigcup\limits_{\lambda\in\Gamma}}
E_{\lambda}\text{.}%
\]
Therefore, $\left(  f\left(  x_{j_{k}}\right)  \right)  _{k=1}^{\infty}\notin%
{\displaystyle\bigcup\limits_{\lambda\in\Gamma}}
E_{\lambda}$ and so $x^{0}\in G\left(  E,f,\left(  E_{\lambda}\right)
_{\lambda\in\Gamma}\right)  $.

It follows from Lemma \ref{Lema2.2} that exists $\mathbb{N}_{1}\subset
\mathbb{N}$ such that
\[
\operatorname*{card}\left(  \mathbb{N}_{1}\right)  =\operatorname*{card}%
\left(  \mathbb{N}\backslash\mathbb{N}_{1}\right)  =\aleph_{0}%
\]
and
\[
\left(  f\left(  x_{j}\right)  \right)  _{j\in\mathbb{N}_{1}}\notin%
{\displaystyle\bigcup\limits_{\lambda\in\Gamma}}
E_{\lambda}\text{.}%
\]
Let us consider some sequence $\left(  \mathbb{N}_{i}\right)  _{i=2}^{\infty}$
of countable and pairwise disjoint subsets of $\mathbb{N}$ such that
\[
\mathbb{N}\backslash\mathbb{N}_{1}\mathbb{=}%
{\displaystyle\bigcup\limits_{i=2}^{\infty}}
\mathbb{N}_{i}.
\]
Denoting
\[
\mathbb{N}_{i}=\left\{  i_{1}<i_{2}<i_{3}<\cdots\right\}  \text{,}%
\]
for $i=1,2\ldots$, let us consider the sequence $\left(  y_{i}\right)
_{i=1}^{\infty}$ defined by%
\[
y_{1}=x=%
{\textstyle\sum\limits_{k=1}^{\infty}}
x_{k}\otimes e_{k}%
\]
and, whenever $i\geq2$,%
\[
y_{i}=%
{\textstyle\sum\limits_{k=1}^{\infty}}
x_{k}\otimes e_{i_{k}}\text{.}%
\]
Notice that $y_{i}^{0}=x^{0}$ and hence $0\neq y_{i}^{0}\in E$ for all $i$.
Since $E$ is an invariant sequence space, it follows that $y_{i}\in E$ for all
$i\in\mathbb{N}$. Also note that the set $\left\{  y_{i}:i\in\mathbb{N}%
\right\}  $ is linearly independent. In fact, let $\lambda_{1},\ldots
,\lambda_{k}\in\mathbb{K}$ be such that%
\[
\lambda_{1}y_{1}+\lambda_{2}y_{2}+\cdots+\lambda_{k}y_{k}=0
\]
We can see that, for all $m\in\mathbb{N}$, the $1_{m}$-th coordinate of
\[
\lambda_{1}y_{1}+\lambda_{2}y_{2}+\cdots+\lambda_{k}y_{k}%
\]
is
\[
\lambda_{1}x_{1_{m}}=0\text{.}%
\]
and, since $m$ can be chosen satisfying $x_{1_{m}}\neq0$, we have $\lambda
_{1}=0$. It follows that
\[
\lambda_{2}y_{2}+\cdots+\lambda_{k}y_{k}=0\text{.}%
\]
Let $m\in\mathbb{N}$ be such that $x_{m}\neq0$. For all $i=2,\ldots,k$, the
$i_{m}$-th coordinate of
\[
\lambda_{2}y_{2}+\cdots+\lambda_{k}y_{k}%
\]
is%
\[
\lambda_{i}x_{m}=0
\]
and so $\lambda_{2}=\cdots=\lambda_{k}=0$. Furthermore, $y_{i}\in G\left(
E,f,\left(  E_{\lambda}\right)  _{\lambda\in\Gamma}\right)  $. Indeed, since
$y_{i}^{0}=x^{0}$, denoting $y_{i}=\left(  y_{m}^{\left(  i\right)  }\right)
_{m=1}^{\infty}$, we have%
\[
\left[  \left(  f\left(  y_{m}^{\left(  i\right)  }\right)  \right)
_{m=1}^{\infty}\right]  ^{0}=\left[  \left(  f\left(  x_{j}\right)  \right)
_{j=1}^{\infty}\right]  ^{0}\notin E_{\lambda}\text{,}%
\]
for all $i\in\mathbb{N}$ and all $\lambda\in\Gamma$. Let $K$ be the constant
of Definition \ref{Def3.1.1}(b1) and let us take $\tilde{s}=1$ if $E$ is a
Banach space and $\tilde{s}=s$ if $E$ is a quasi-Banach space, $0<s<1$. For
each $\left(  a_{i}\right)  _{i=1}^{\infty}\in\ell_{\tilde{s}}$,%
\begin{align*}%
{\textstyle\sum\limits_{i=1}^{\infty}}
\left\Vert a_{i}y_{i}\right\Vert _{E}^{\tilde{s}}  &  =%
{\textstyle\sum\limits_{i=1}^{\infty}}
\left\vert a_{i}\right\vert ^{\tilde{s}}\cdot\left\Vert y_{i}\right\Vert
_{E}^{\tilde{s}}\leq K^{\tilde{s}}%
{\textstyle\sum\limits_{i=1}^{\infty}}
\left\vert a_{i}\right\vert ^{\tilde{s}}\cdot\left\Vert y_{i}^{0}\right\Vert
_{E}^{\tilde{s}}\\
&  =K^{\tilde{s}}\left\Vert x^{0}\right\Vert _{E}^{\tilde{s}}%
{\textstyle\sum\limits_{i=1}^{\infty}}
\left\vert a_{i}\right\vert ^{\tilde{s}}<\infty\text{.}%
\end{align*}
Hence, $%
{\textstyle\sum\limits_{i=1}^{\infty}}
\left\Vert a_{i}y_{i}\right\Vert _{E}^{\tilde{s}}<\infty$ and, in both cases,
we conclude that $%
{\textstyle\sum\limits_{i=1}^{\infty}}
a_{i}y_{i}$ converges in $E$. Therefore, the operator
\begin{align*}
T\colon\ell_{\tilde{s}}  &  \rightarrow E\\
\left(  a_{k}\right)  _{k=1}^{\infty}  &  \mapsto T\left(  \left(
a_{k}\right)  _{k=1}^{\infty}\right)  =%
{\textstyle\sum\limits_{k=1}^{\infty}}
a_{k}y_{k}%
\end{align*}
is well-defined and linear. Let us see that $T$ is injective. In fact, let
$\left(  \mu_{n}\right)  _{n=1}^{\infty}\in\ell_{\tilde{s}}$ be such that%
\[
T\left(  \left(  \mu_{k}\right)  _{k=1}^{\infty}\right)  =0\text{.}%
\]
Since $f\left(  0\right)  =0$ and $\left(  f\left(  x_{1_{k}}\right)  \right)
_{k=1}^{\infty}\notin%
{\displaystyle\bigcup\limits_{\lambda\in\Gamma}}
E_{\lambda}$, there is $m\in\mathbb{N}$ such that $x_{1_{m}}\neq0$; but the
$1_{m}$-th coordinate of $T\left(  \left(  \mu_{k}\right)  _{k=1}^{\infty
}\right)  $ is $\mu_{1}x_{1_{m}}$ and we conclude that $\mu_{1}=0$. Now, if we
fix $p\in\mathbb{N}$ such that $x_{p}\neq0$, we have, for $i\geq2$, that the
$i_{p}$-th coordinate of $T\left(  \left(  \mu_{k}\right)  _{k=1}^{\infty
}\right)  $ is $\mu_{i}x_{p}$ and we conclude that $\mu_{i}=0$ for all
$i\geq2$.

Recalling that $x=y_{1}\in T\left(  \ell_{\tilde{s}}\right)  $, let us show
that $\overline{T\left(  \ell_{\tilde{s}}\right)  }\subset\left(  G\left(
E,f,\left(  E_{\lambda}\right)  _{\lambda\in\Gamma}\right)  \cup\left\{
0\right\}  \right)  $. In order to do this, we will check that, if $z=\left(
z_{n}\right)  _{n=1}^{\infty}\in\overline{T\left(  \ell_{\tilde{s}}\right)  }$
is a non-zero sequence, then $\left(  f\left(  z_{n}\right)  \right)
_{n=1}^{\infty}\notin%
{\displaystyle\bigcup\limits_{\lambda\in\Gamma}}
E_{\lambda}$. Let $\left(  a_{i}^{\left(  k\right)  }\right)  _{i=1}^{\infty
}\in\ell_{\tilde{s}}$, $k\in\mathbb{N}$, be such that $z=\lim
\limits_{k\rightarrow\infty}T\left(  \left(  a_{i}^{\left(  k\right)
}\right)  _{i=1}^{\infty}\right)  $ in $E$. Notice that, for each
$k\in\mathbb{N}$,%
\begin{align*}
T\left(  \left(  a_{i}^{\left(  k\right)  }\right)  _{i=1}^{\infty}\right)
&  =%
{\textstyle\sum\limits_{i=1}^{\infty}}
a_{i}^{\left(  k\right)  }y_{i}\\
&  =a_{1}^{\left(  k\right)  }%
{\textstyle\sum\limits_{k=1}^{\infty}}
x_{k}\otimes e_{k}+%
{\textstyle\sum\limits_{i=2}^{\infty}}
a_{i}^{\left(  k\right)  }%
{\textstyle\sum\limits_{k=1}^{\infty}}
x_{k}\otimes e_{i_{k}}\\
&  =%
{\textstyle\sum\limits_{k=1}^{\infty}}
a_{1}^{\left(  k\right)  }x_{k}\otimes e_{k}+%
{\textstyle\sum\limits_{i=2}^{\infty}}
{\textstyle\sum\limits_{k=1}^{\infty}}
a_{i}^{\left(  k\right)  }x_{k}\otimes e_{i_{k}}\\
&  =\left(  w_{r}^{\left(  k\right)  }\right)  _{r=1}^{\infty}\text{,}%
\end{align*}
where
\[
w_{r}^{\left(  k\right)  }=\left\{
\begin{array}
[c]{ll}%
a_{1}^{\left(  k\right)  }x_{1_{m}}\text{,} & \text{\hspace{-0.25cm}if
}r=1_{m}\in\mathbb{N}_{1}\text{,}\\
a_{1}^{\left(  k\right)  }x_{i_{m}}+a_{i}^{\left(  k\right)  }x_{m}\text{,} &
\text{\hspace{-0.25cm}if }r=i_{m}\in\mathbb{N}_{i}\text{ and }i\geq2\text{.}%
\end{array}
\right.
\]
The condition (b2) of Definition \ref{Def3.1.1} assures that convergence in
$E$ implies coordinatewise convergence. Fixed some $p\in\mathbb{N}$ such that
$x_{1_{p}}\neq0$, we have%
\[
z_{1_{p}}=\lim\limits_{k\rightarrow\infty}a_{1}^{\left(  k\right)  }x_{1_{p}%
}=\left(  \lim\limits_{k\rightarrow\infty}a_{1}^{\left(  k\right)  }\right)
x_{1_{p}}\text{,}%
\]
that is, $\lim\limits_{k\rightarrow\infty}a_{1}^{\left(  k\right)  }$ exists.
Let $a_{1}=\lim\limits_{k\rightarrow\infty}a_{1}^{\left(  k\right)  }$. On the
other hand, if $n=i_{p}$, $i\geq2$, where $p$ is chosen satisfying $x_{p}%
\neq0$, then%
\[
z_{i_{p}}=\lim\limits_{k\rightarrow\infty}\left(  a_{1}^{\left(  k\right)
}x_{i_{p}}+a_{i}^{\left(  k\right)  }x_{p}\right)  \text{.}%
\]
Thus,%
\begin{align*}
\left(  \lim\limits_{k\rightarrow\infty}a_{i}^{\left(  k\right)  }\right)
x_{p}  &  =\lim\limits_{k\rightarrow\infty}a_{i}^{\left(  k\right)  }x_{p}\\
&  =\lim\limits_{k\rightarrow\infty}\left[  \left(  a_{1}^{\left(  k\right)
}x_{i_{p}}+a_{i}^{\left(  k\right)  }x_{p}\right)  -a_{1}^{\left(  k\right)
}x_{i_{p}}\right] \\
&  =\lim\limits_{k\rightarrow\infty}\left(  a_{1}^{\left(  k\right)  }%
x_{i_{p}}+a_{i}^{\left(  k\right)  }x_{p}\right)  -\lim\limits_{k\rightarrow
\infty}a_{1}^{\left(  k\right)  }x_{i_{p}}\\
&  =z_{i_{p}}-a_{1}x_{i_{p}}\text{,}%
\end{align*}
that is, $\lim\limits_{k\rightarrow\infty}a_{i}^{\left(  k\right)  }$ exists
and it will be denoted by $a_{i}$. Therefore, coordinatewise convergence
yields%
\[
z_{n}=\left\{
\begin{array}
[c]{ll}%
a_{1}x_{1_{m}}\text{,} & \text{\hspace{-0.25cm}if }n=1_{m}\in\mathbb{N}%
_{1}\text{,}\\
a_{1}x_{i_{m}}+a_{i}x_{m}\text{,} & \text{\hspace{-0.25cm}if }n=i_{m}%
\in\mathbb{N}_{i}\text{ and }i\geq2\text{.}%
\end{array}
\right.
\]
Since $\left(  z_{n}\right)  _{n=1}^{\infty}$ is non-zero is immediate that
$a_{i}\neq0$ for some $i\in\mathbb{N}$.

If $a_{1}\neq0$, we have $z_{1_{m}}=a_{1}x_{1_{m}}$ for all $m\in\mathbb{N}$.
Note that
\[
\left(  f\left(  z_{1_{m}}\right)  \right)  _{m=1}^{\infty}=\left(  f\left(
a_{1}x_{1_{m}}\right)  \right)  _{m=1}^{\infty}=\left(  f\left(  a_{1}%
x_{j}\right)  \right)  _{j\in\mathbb{N}_{1}}\text{.}%
\]
Since, for all $\lambda\in\Gamma$, $f$ is compatible with $E_{\lambda}$ and
$\left(  f\left(  x_{j}\right)  \right)  _{j\in\mathbb{N}_{1}}\notin
E_{\lambda}$, we have%
\[
\left(  f\left(  z_{1_{m}}\right)  \right)  _{m=1}^{\infty}\notin E_{\lambda}%
\]
for all $\lambda$. Hence, there exists a subsequence of $\left(  f\left(
z_{n}\right)  \right)  _{n=1}^{\infty}$ not belonging to $E_{\lambda}$ for all
$\lambda$. Since $E_{\lambda}$ is a strongly invariant sequence space for all
$\lambda\in\Gamma$, it follows that%
\[
\left(  f\left(  z_{n}\right)  \right)  _{n=1}^{\infty}\notin E_{\lambda
}\text{,}%
\]
for all $\lambda\notin\Gamma$.

If $a_{1}=0$ and $a_{p}\neq0$ for some $p>1$, then $z_{p_{m}}=a_{p}x_{m}$ for
all $m\in\mathbb{N}$ and, in this case, $\left(  f\left(  z_{p_{m}}\right)
\right)  _{m=1}^{\infty}=\left(  f\left(  a_{p}x_{m}\right)  \right)
_{m=1}^{\infty}$. An analogous argument allows us to conclude that%
\[
\left(  f\left(  z_{n}\right)  \right)  _{n=1}^{\infty}\notin E_{\lambda
}\text{.}%
\]
This completes the proof that $z\in G\left(  E,f,\left(  E_{\lambda}\right)
_{\lambda\in\Gamma}\right)  $ and therefore $G\left(  E,f,\left(  E_{\lambda
}\right)  _{\lambda\in\Gamma}\right)  $ is pointwise $\mathfrak{c}$-spaceable.
\end{proof}

We stress that the previous theorem does not cover results of the type
(\ref{s44}), but this can be done with similar arguments as follows.

If $F$ is an invariant sequence space over the scalar field $\mathbb{K}$ and
$Y$ is a Banach space, we define
\[
F^{w}\left(  Y\right)  :=\left\{  \left(  x_{j}\right)  _{j=1}^{\infty}\in
Y^{\mathbb{N}}:\left(  \varphi\left(  x_{j}\right)  \right)  _{j=1}^{\infty
}\in F\text{ for all }\varphi\in Y^{\prime}\right\}
\]
and note that
\[
\sup_{\left\Vert \varphi\right\Vert \leq1}\left\Vert \left(  \varphi\left(
x_{j}\right)  \right)  _{j=1}^{\infty}\right\Vert _{F}<\infty
\]
for all $\left(  x_{j}\right)  _{j=1}^{\infty}\in F^{w}\left(  Y\right)  $
(see \cite[p. 178]{NP}).

If $F=\ell_{p}$, the respective space $F^{w}\left(  Y\right)  $ is the
well-known invariant sequence space $\ell_{p}^{w}\left(  Y\right)  $.
According to \cite[Definition 3.3]{NP}, if $X$ and $Y$ are Banach spaces and
$F$ is an invariant sequence space over $\mathbb{K}$, then a map $f\colon
X\rightarrow Y$ such that $f\left(  0\right)  =0$ is called strongly
compatible with $F^{w}\left(  Y\right)  $ if $\varphi\circ f$ is compatible
with $F$ for all continuous linear functionals $\varphi\colon Y\rightarrow
\mathbb{K}$. It is obvious that any strongly non-contractive map $f\colon
X\rightarrow Y$ is strongly compatible with $\ell_{p}^{w}\left(  Y\right)  $
and with%
\[
c_{0}^{w}\left(  Y\right)  :=\left\{  \left(  y_{j}\right)  _{j=1}^{\infty}%
\in\ell_{\infty}\left(  Y\right)  :\lim\limits_{j\rightarrow\infty}%
\varphi\left(  y_{j}\right)  =0\text{, for all }\varphi\in Y^{\prime}\right\}
\text{.}%
\]
The following result was proved in \cite[Theorem 2.5]{NP}:

\begin{theorem}
\label{Teo3.3.6}Let $X$ and $Y$ be Banach spaces, $\Gamma$ an arbitrary set,
$E$ an invariant sequence space over $X$ and, for all $\lambda\in\Gamma$, let
$F_{\lambda}$ be a strongly invariant sequence space over $\mathbb{K}$. If
$f\colon X\rightarrow Y$ is strongly compatible with $F_{\lambda}^{w}\left(
Y\right)  $ for all $\lambda\in\Gamma$, then
\[
G^{w}\left(  E,f,\left(  F_{\lambda}\right)  _{\lambda\in\Gamma}\right)
:=\left\{  \left(  x_{j}\right)  _{j=1}^{\infty}\in E:\left(  f\left(
x_{j}\right)  \right)  _{j=1}^{\infty}\notin%
{\displaystyle\bigcup\limits_{\lambda\in\Gamma}}
F_{\lambda}^{w}\left(  Y\right)  \right\}
\]
is either empty or spaceable.
\end{theorem}

The proof of Theorem \ref{Teo1} can be adapted, \textit{mutatis mutandis}, to
obtain the following result, which shows that Theorem \ref{Teo3.3.6} is valid
for pointwise spaceability.

\begin{theorem}
Let $X$ and $Y$ be Banach spaces, $\Gamma$ an arbitrary set, $E$ an invariant
sequence space over $X$ and, for all $\lambda\in\Gamma$, let $F_{\lambda}$ be
a strongly invariant sequence space over $\mathbb{K}$. If $f\colon
X\rightarrow Y$ is strongly compatible with $F_{\lambda}^{w}\left(  Y\right)
$ for all $\lambda\in\Gamma$ and $\left(  F_{\lambda}^{w}\left(  Y\right)
\right)  _{\lambda\in\Gamma}$ is nested, then $G^{w}\left(  E,f,\left(
F_{\lambda}\right)  _{\lambda\in\Gamma}\right)  $ is either empty or pointwise
$\mathfrak{c}$-spaceable.
\end{theorem}

\end{document}